\documentclass[a4paper,10pt]{article}
\usepackage{mathtext}
\usepackage[T1,T2A]{fontenc}
\usepackage[cp1251]{inputenc}
\usepackage[english]{babel}
\usepackage{amsmath}
\usepackage{amsfonts}
\usepackage{amssymb}
\usepackage{mathrsfs}
\usepackage{amsthm}
\usepackage{enumerate}
\usepackage{graphicx}
\graphicspath{}
\DeclareGraphicsExtensions{.pdf,.png,.jpg}

\usepackage{color}
\usepackage{euscript}

\usepackage{cite}

\newtheorem{Le}{Lemma}[section]
\newtheorem{Def}{Definition}[section]

\newtheorem{Th}{Theorem}[section]

\numberwithin{equation}{section}

\newcommand{\R}{\mathbb{R}}
\newcommand{\Co}{\mathbb{C}}

\newcommand{\eps}{\varepsilon}

\newcommand{\eq}[1]{\begin{equation}{#1}\end{equation}}

\newcommand{\Set}[2]{\Big\{\ {#1}\ \,\Big|\;{#2}\Big\}}

\DeclareMathOperator{\BV}{BV}

\DeclareMathOperator{\supp}{supp}

\DeclareMathOperator{\diam}{diam}

\title{Frostman lemma revisited}
\author{Nikita Dobronravov\footnote{Supported by Theoretical Physics and Mathematics Advancement Foundation “BASIS” grant Junior Leader (Math) 21-7-2-12-2}}
\begin{document}
\maketitle
\begin{abstract}
We study sharpness of various generalizations of Frostman's lemma. These generalizations provide better estimates for the lower Hausdorff dimension of measures. As a corollary, we prove that if a generalized anisotropic gradient~$(\partial_1^{m_1} f, \partial_2^{m_2} f,\ldots, \partial_d^{m_d} f)$ of a function~$f$ in~$d$ variables is a measure of bounded variation, then this measure is absolutely continuous with respect to the Hausdorff~$d-1$ dimensional measure.
\end{abstract}

\section{Introducion}
The classical Frostman lemma says that for any compact set~$F \subset \R^d$, its~$\alpha$-Hausdorff measure is not zero if and only if there exists a finite non-zero Borel measure~$\mu$ supported on~$F$ such that for any open ball~$B_r(x)$ of radius~$r$ centered at~$x\in \R^d$ the inequality
\eq{
\mu(B_r(x)) \leq r^\alpha
}

holds true. See~\cite{Mattila1995} for details (Theorem 8.8). Recall the definition of the lower Hausdorff dimension of a measure. Let~$\mu$ be a possibly vector-valued (with values in a finite dimensional vector space over~$\R$ or~$\Co$) locally finite Borel measure on~$\R^d$. The said definition reads as follows:
\eq{
\dim_{\mathrm{H}} \mu = \inf\Set{\alpha}{ \hbox{ there exists a Borel set } F \hbox{ such that } \dim_{\mathrm{H}} F \leq \alpha \hbox{ and } \mu(F) \ne 0}.
}
Since we will be working with the Hausdorff dimension only, we suppress the indices~$\mathrm{H}$ in our notation. If a measure~$\mu$ satisfies the estimate
\eq{\label{SignedFrostman}
|\mu(B_r(x))| \lesssim r^\alpha
}
for all~$x\in \R^d$ and~$r > 0$ uniformly, then the Frostman lemma implies~$\dim \mu \geq \alpha$. Here and in what follows the notation~$A \lesssim B$ means there exists a uniform constant~$C$ such that~$A \leq CB$. In particular, the constant in~\eqref{SignedFrostman} should depend neither on~$x$ nor on~$r$.  Note that here we use a simpler-to-prove implication in the Frostman lemma. 

The inequality~\eqref{SignedFrostman} is not equivalent to the assertion~$\dim \mu \geq \alpha$ or any small perturbation of it. A good example is the measure~$|x|^{-\beta}\,dx$, $\beta<1$, on the line, which has dimension~$1$, but violates~\eqref{SignedFrostman} for~$\alpha > 1-\beta$. On the other hand, the assertion~$\dim \mu \geq \alpha$ is equivalent to the statement that for~$\mu$ almost every~$x$ and any~$\eps > 0$ the inequality
\eq{
|\mu|(B_r(x)) \leq r^{\alpha - \eps}
}
holds true for all sufficiently small~$r$ (depending on~$\eps$ and~$x$), see~\cite{Falconer1997} (Proposition 10.2). The latter local condition is quite difficult to handle. It is desirable to provide a more uniform global one. 

Let~$\mu$ be a non-negative scalar measure. Consider the energy integral
\eq{
\int\limits_{\R^{2d}}\frac{d\mu(x)d\mu(y)}{|x-y|^{\alpha}}.
}
If the energy integral converges, then~$\dim \mu \geq \alpha$. The energy method may be efficient for some problems, however, has its limitations (see~\cite{Mattila2015}). The next lemma provides another uniform condition sufficient for~$\dim \mu \geq \alpha$.

\begin{Le}[Lemma~$1$ in \cite{StolyarovWojciechowski2014}]\label{SWFrostman}
Suppose that~$\varphi \in C^{\infty}_0(\mathbb{R}^d)$ is a radial non-negative function supported in a unit ball. Assume that~$\varphi$ decreases as the radius grows and~$\varphi(x) = 1$ when~$|x| \leq\frac34$.  
Let~$\mu$ be a measure such that for every collection~$B_{r_j}(x_j)$ of~$d$-dimensional balls such that~$B_{3r_j}(x_j)$ are disjoint the estimate
\begin{equation}\label{SWFrostmanIneq}
\sum\limits_{j}\Big|\int\limits_{\mathbb{R}^d}\varphi_{3r_j}(x_j + y)\,d\mu(y)\Big| \lesssim \big(\sum\limits r_j^{\alpha}\big)^{\beta}
\end{equation}
holds true for some positive~$\alpha$ and~$\beta$. Then~$\dim(\mu) \geq \alpha$.
\end{Le}
Here $\varphi_t(z)=\varphi\left(\frac{z}{t}\right)$.

In the case~$\beta = 1$, the above lemma reduces to the classical Frostman lemma, however, in the case~$\beta < 1$ it is a priori stronger (Lemma~\ref{justification} below provides examples). In~\cite{StolyarovWojciechowski2014}, the lemma served as a technical tool to prove that any vector-valued measure of the form~$(\partial_1^{m_1} f, \partial_2^{m_2} f,\ldots, \partial_d^{m_d} f)$ has dimension at least~$d-1$. Here~$\partial_j$ is the operator of differentiation with respect to~$j$-th coordinate and~$m_1, m_2,\ldots, m_d$ are arbitrary natural numbers. Note that the case~$m_1=m_2=\ldots =m_d = 1$ may be deduced from the co-area formula for~$\BV$ functions (see~\cite{AFP2000} for this formula).

In recent years, there is an increasing interest in the geometry of measures satisfying PDE or Fourier constraints, like gradients of functions of bounded variation, or  divergence free measures, or, for example, the generalized gradient measures as above. We refer the reader to the papers~\cite{APHF2018},~\cite{DePhilippisRindler2016}, and~\cite{RoginskayaWojciechowski2006}, to mention a few.  Lemma~\ref{SWFrostman} appeared useful in this context, in particular, an analog of this lemma plays the pivotal role in~\cite{ASW2018}, where a simpler discrete analog of the dimension problem for Fourier constrained measures is solved. A similar lemma provides good (better than the ones given by the energy method) dimensional estimates for Riesz products, see~\cite{ASW2020}. The purpose of this article is to sharpen and generalize Lemma~\ref{SWFrostman}. These considerations will also lead to a corollary about generalized gradients (see Theorem~\ref{Stol} below).

I am grateful to my scientific adviser D. M. Stolyarov for statement of the problem and attention to my work.

\section{Statement of results}
There are several essences involved in Lemma~\ref{SWFrostman}. The function~$\varphi$ is, in a sense, auxiliary. It allows to replace the expression~$\mu(B_r(x))$ with a smoother one. Namely, the integral
\eq{
\int\limits_{\mathbb{R}^d}\varphi_{3r}(x + y)\,d\mu(y)
} 
may be thought of as a smoothing of~$\mu(B_r(x))$. So, the expression on the left hand side in~\eqref{SWFrostmanIneq} may be thought of as the sum~$\sum_j |\mu(B_{r_j}(x_j))|$.  The parameter~$\beta$ is quite mysterious since the estimate~$\dim \mu \geq \alpha$ does not depend on~$\beta$. We prefer to replace the number~$\beta$ with a function, i.e. to estimate the left hand side of~\eqref{SWFrostmanIneq} with~$g(\sum_j r_j^\alpha)$, where~$g$ is a certain weight function; in Lemma~\ref{SWFrostman},~$g(r) = r^{\beta}$. This generalization seems reasonable in the light of the entropy estimates in~\cite{ASW2018}. 
\begin{Def}\label{regyl}
A continuous function~$g\colon\mathbb{R}_+\rightarrow \mathbb{R}_+$ is called regular provided~$g(0)=0$,~$g$ does not decrease, and for some fixed~$c > 1$ we have~$g(x)\asymp g(cx)$.
\end{Def}
Here and in what follows~$A \asymp B$ is short for~$A \lesssim B$ and~$B \lesssim A$. We say that the functions~$\varphi$ and~$\psi$ are equivalent if~$\varphi(x) \asymp \psi(x)$ for any~$x\in \R^d$.

We are ready to formulate our first result. By a disjoint family of balls we mean a collection of balls non of which intersect.
\begin{Le}\label{teor1}
Let~$\varphi$ be a radially symmetric, radially non-increasing function supported in the unit ball. Assume also~$\varphi(x)=1$ when~$|x|\leqslant \frac{3}{4}$. Let~$\psi$ be equivalent to~$\varphi$. Let~$\mu$ be an~$\R$-valued signed measure of locally bounded variation, let~$g$ be a regular function that satisfies the Dini condition
\eq{
\int\limits_0^1\frac{g(t)}{t}dt<\infty.
}
Assume that  
\eq{
\sum\limits_{B_{r_j}(x_j)\in \mathfrak{B}}\left|\underset{\mathbb{R}^d}{\int}\psi\left(\frac{y-x_j}{r_j} \right) d\mu(y) \right|\lesssim g\Big(\sum\limits_{B_{r_j}(x_j)\in \mathfrak{B}}r_j^\alpha\Big)
}
for any disjoint family of balls~$\mathfrak{B}$. Then,
\eq{
|\mu|(A)\lesssim h(\mathcal{H}^{\alpha}(A)),
}
where
\eq{
h(x)=\int\limits_0^x\frac{g(t)}{t}dt
}
and~$A \subset \R^d$ is an arbitrary Borel set.
\end{Le}

It is easy to see that for regular $g$ there is the inequality $g\leqslant h$.
In particular,~$\mu$ is absolutely continuous with respect to~$\mathcal{H}^\alpha$, and Lemma~\ref{SWFrostman} is a corollary of Lemma~\ref{teor1}. In the case~$\alpha > d-1$, we are able to get rid of the condition that~$\varphi$ is radial. 
\begin{Le}\label{teor2}
Let~$\varphi$ be a bounded function supported in the unit ball and such that~$\varphi(x)\geqslant1$ when~$|x|\leqslant \frac{3}{4}$. Let~$\mu$ be an~$\R$-valued signed measure of locally bounded variation and let~$g$ be a regular function. Assume also that~$\alpha > d-1$. If
\eq{
\sum\limits_{B_{r_j}(x_j)\in \mathfrak{B}}\left|\underset{\mathbb{R}^d}{\int}\varphi\left(\frac{y-x_j}{r_j} \right) d\mu(y) \right|\lesssim g\Big(\sum\limits_{B_{r_j}(x_j)\in \mathfrak{B}}r_j^\alpha\Big)
}
for any disjoint family of balls~$\mathfrak{B}$, then
\eq{
|\mu|(A)\lesssim g(\mathcal{H}^{\alpha}(A))
}
for any Borel set~$A \subset \R^d$.
\end{Le}
Though lemmas~\ref{teor1} and~\ref{teor2} look similar, their proofs are quite different. Lemma~\ref{teor2} has a nice corollary.
\begin{Th}\label{Stol}
Let~$m_1, m_2, \ldots, m_d$ be natural numbers. Assume~$f$ is a distribution such that for any~$j = 1,2,\ldots,d$ the distribution~$\partial^{m_j}_j f$ is a signed measure of bounded variation. Then, these signed measures are absolutely continuous with respect to~$\mathcal{H}^{d-1}$.
\end{Th}
We also provide a version of Lemma~\ref{SWFrostman} where the function~$\varphi$ is not compactly supported. This version is surprisingly easier to prove.
\begin{Le}\label{teor3}
Let~$\varphi$ be a bounded radially symmetric, radially non-increasing function such that~$B_3(0)\subset\supp \varphi $. Let~$\psi$  be equivalent to~$\varphi$. Let~$\mu$ be a signed measure of bounded variation, let~$g$ be a regular function. Assume
\eq{
\sum\limits_{B_{r_j}(x_j)\in \mathfrak{B}}\left|\underset{\mathbb{R}^d}{\int}\psi\left(\frac{y-x_j}{r_j} \right) d\mu(y) \right|\lesssim g\Big(\sum\limits_{B_{r_j}(x_j)\in \mathfrak{B}}r_j^\alpha\Big)
}
for any disjoint family of balls~$\mathfrak{B}$. If
\eq{
|\varphi(x)|=O(|x|^{-\alpha}), \ \ x\rightarrow \infty,
}
then~$\mu$ is absolutely continuous with respect to~$\mathcal{H}^{\alpha}$. If
\eq{
|\varphi(x)| = o(|x|^{-\alpha}),\quad x\to \infty,
}
then the inequality
\eq{
|\mu|(A)\lesssim g(\mathcal{H}^{\alpha}(A))
}
holds true for any Borel set~$A$.
 \end{Le}

The proofs of lemmas~\ref{teor1} and~\ref{teor2} are based on new covering lemmas, which may be interesting in themselves. They deal with a simple notion that seems to be important for working with local properties of signed measures.
\begin{Def}\label{super}
A family~$\mathfrak{B}$ of open Euclidean balls is called a supercovering of a set~$A\subset \R^d$, provided
\eq{\label{superpokr}
A\subset\!\!\!\bigcup\limits_{B_{r_j}(x_j)\in\mathfrak{B}}\!\!\! B_{\frac{r_j}{3}}(x_j).
}
\end{Def}

\begin{Le}\label{super1}
Let~$\alpha \in [0,d]$. There exist the constants~$q(\alpha) \in (0,1)$ and~$C(d,\alpha)>0$ such that for any compact set~$K\subset\mathbb{R}^d$,~$\mathcal{H}^{\alpha}(K)<a$, and any~$\eps > 0$ there exists a finite supercovering~$\mathfrak{B}$ satisfying the requirements

\begin{enumerate}[1)]
\item the center of any ball of~$\mathfrak{B}$ lies inside~$K$ and its radius does not exceed~$\varepsilon$,

\item there exists a natural number $N$ such that~$\mathfrak{B}$ may be split into~$N$ disjoint subfamilies~$\mathfrak{B}^j$ satisfying the bound
\eq{
\sum\limits_{B_{r_i}(x_i)\in\mathfrak{B}^j}r_i^{\alpha}\leqslant Cq^ja, \ \ \  \ \  \ j=1,\ 2,\ \dots ,\ N.
}
\end{enumerate}
\end{Le}

Rem. Lemma~\ref{super1} does not provide any control on $N$.

\begin{Le}\label{super2}
Let~$\alpha\in (d-1,d]$. There exist the constants~$C_1(d)>0$ and~$C_2(d,\alpha)>0$ such that for any bounded set~$A\subset\mathbb{R}^d$,~$\mathcal{H}^{\alpha}(A)<a$, and any~$\eps > 0$, there exists a supercovering~$\mathfrak{B}$, satisfying the requirements
\begin{enumerate}[1)]
\item the center of any ball of~$\mathfrak{B}$ lies inside~$K$ and its radius does not exceed~$\varepsilon$,

\item the family~$\mathfrak{B}$ may be split into~$C_1$ disjoint subfamilies~$\mathfrak{B}^j$ such that
\eq{
\sum\limits_{B_{r_i}(x_i)\in\mathfrak{B}^j}r_i^{\alpha}\leqslant C_2a,  \ \ \  \ \  \ j=1,\ 2,\ \dots ,\ C_1.
}
\end{enumerate}
\end{Le}
The constants in the lemmas admit explicit expressions. The proofs of lemmas~\ref{super1} and~\ref{super2} are presented in Section~\ref{S2}. The proofs of lemmas~\ref{teor1} and~\ref{teor2} as well as Theorem~\ref{Stol} may be found in Section~\ref{S3}. The last section contains some generalizations of these results and some examples.

 \section{Constructions of fine coverings}\label{S2}
 We need some tools.

 		
 		
 		
 		
	


\begin{Le}\label{Morse1}
	For any $d\in\mathbb{N}$ there exists a constant $\theta(d)$ such that the following holds. Let $A$ be a bounded subset of $\mathbb{R}^d$ and let $\mathfrak{B}$ be a family of balls in~$\mathbb{R}^d$.
	Assume that
	\eq{\forall x\in A \ \ \ \exists B_r(y)\in \mathfrak{B} \ \ \  such \ that \ x\in B_{\frac{r}{3}}(y).}
	Then there exists a subfamily $\mathfrak{B}'\subset\mathfrak{B}$ satisfying the requirements
	\begin{enumerate}[1)]
		\item $A\subset \underset{B_{r_i}(x_i)\in\mathfrak{B}'}{\bigcup}B_{r_i}(x_i)$,
		
		\item the family $\mathfrak{B}'$ may be split into $\theta$ disjoint subfamilies.
	\end{enumerate}
\end{Le}
Lemma~\ref{Morse1} is a particular case of the Morse covering theorem (see~p.6~in~\cite{Guzman1975}).







\begin{Le}\label{pok1}
	Let $\alpha\in[0,d]$. There exist the constants $C(\alpha)$ and $\theta(d)$ such that for any bounded $A\subset\mathbb{R}^d$, $\mathcal{H}^{\alpha}(A)<a$ (for some constant $a$), and any $\varepsilon>0$ there exists a family of balls $\mathfrak{B}$ such that
	\begin{enumerate}[1)]
		\item $\mathfrak{B}$ is a covering of $A$,
		\item the center of any ball in~$\mathfrak{B}$ lies in~$A$ and its radius does not exceed~$\varepsilon$,
		
		\item the family~$\mathfrak{B}$ may be split into~$\theta$ disjoint subfamilies~$\mathfrak{B}^j$ such that
		\eq{
			\sum\limits_{B_{r_i}(x_i)\in\mathfrak{B}^j}r_i^{\alpha}\leqslant Ca.
		}
	\end{enumerate}
\end{Le}
\begin{proof}
	Since $\mathcal{H}^{\alpha}(A)<a$ there exists a family $\mathfrak{A}$ such that $\mathfrak{A}$ is a covering $A$, for any $D\in\mathfrak{A}$ we have $\diam(D)<\frac{\eps}{1000}$, and
	\eq{\sum_{D_i\in\mathfrak{A}}\diam(D_i)^{\alpha}<a.}
	The family $\mathfrak{B}_0$ is defined in the following way: for any $D\in \mathfrak{A}$ we choose a point $x\in D\cap A$ and put the ball $B_{3\diam(D)}(x)$ into $\mathfrak{B}_0$. The family $\mathfrak{B}_0$ and the set $A$ satisfy the conditions of Lemma~\ref{Morse1}. Let $\mathfrak{B}$ be the subfamily of $\mathfrak{B}_0$ provided by Lemma~\ref{Morse1}. The family $\mathfrak{B}$ splits into $\theta(d)$ disjoint subfamilies $\mathfrak{B}^j$ that satisfy the estimate:
	\eq{
	\sum_{B_{r_i}(x_i)\in\mathfrak{B}^j}r_i^{\alpha}\leqslant\sum_{B_{r_i}(x_i)\in\mathfrak{B}_0}r_i^{\alpha}\leqslant\sum_{D_i\in\mathfrak{A}}(3\diam(D_i))^{\alpha}\leqslant 3^{\alpha}a.
	} 
\end{proof}

\begin{Le}\label{opti}
	Let $K\subset\mathbb{R}^d$ be a compact set such that $\mathcal{H}^{\alpha}(K)<a$ (for some constant $a$) and let $\eps>0$. There exists a constant~$M$ and a family of closed balls $\mathfrak{B}$ such that 
		\begin{enumerate}[1)]
		\item $\mathfrak{B}$ is a covering of $K$,
		
		\item the center of any ball in~$\mathfrak{B}$ lies inside~$K$ and its radius does not exceed~$\varepsilon$,
		
		\item $|\mathfrak{B}|=M$ and
		\eq{\sum_{\overset{\_}{B}_{r_i}(x_i)\in\mathfrak{B}}r_i^{\alpha}<a}
		
		\item For any family $\mathfrak{B}'$ that satisfies 1), 2), 3) the inequality \eq{\sum_{\overset{\_}{B}_{r_i}(x_i)\in\mathfrak{B}}r_i^{\alpha}\leqslant\sum_{\overset{\_}{B}_{r_i}(x_i)\in\mathfrak{B}'}r_i^{\alpha}}
		is true.
		
	\end{enumerate}
\end{Le}
 The constant $M$ depends on $K$ and $\varepsilon$. The 
 set $\overset{\_}{D}$ is the closure of $D$.

\begin{proof}
	 	Since $\mathcal{H}^{\alpha}(A)<a$ there exists a family $\mathfrak{A}$ such that $\mathfrak{A}$ is a covering of $A$, for any $D\in\mathfrak{A} \ $ we have $\diam(D)<\frac{\eps}{1000}$, and
	 \eq{\sum_{D_i\in\mathfrak{A}}\diam(D_i)^{\alpha}<a.}
	 The family $\mathfrak{B}_0$ is defined in the following way: for any $D\in \mathfrak{A}$ we choose a point $x\in D\cap A$ and put the ball $\overset{\_}{B}_{\diam(D)}(x)$ into $\mathfrak{B}_0$. Note that
	 \eq{\sum_{\overset{\_}{B}_{r_i}(x_i)\in\mathfrak{B}_0}r_i^{\alpha}\leqslant\sum_{D_i\in\mathfrak{A}}\diam(D_i)^{\alpha}<a.} 
	 There exists $\delta>0$ such that
	 \eq{\sum_{\overset{\_}{B}_{r_i}(x_i)\in\mathfrak{B}_0}((1+\delta)r_i)^{\alpha}<a.} 
	 If we multiply the radii of the balls $\overset{\_}{B}\in\mathfrak{B}_0$ by $1+\delta$ and make the balls open, then we will get an open covering of  the compact set $K$. Let $\mathfrak{B}_1$ be a finite subcovering of this open covering. Note that we still have 
	 \eq{\sum_{B_{r_i}(x_i)\in\mathfrak{B}_1}r_i^{\alpha}\leqslant\sum_{\overset{\_}{B}_{r_i}(x_i)\in\mathfrak{B}_0}((1+\delta)r_i)^{\alpha}<a.}
	 
	 So if we close the balls in the family $\mathfrak{B}_1$ and take $M=|\mathfrak{B}_1|$, we will have a family that satisfies conditions 1), 2), 3). Now we will prove that there exists an optimal family of this kind. The set $S\subset(\mathbb{R}^d)^M\times\mathbb{R}^M$ is defined by the formula
	 \eq{S=\Big\{(x_1,x_2,\dots,x_M,r_1,\dots,r_M)\mid \ x_j\in K,\ r_j\in[0,\eps], \ K\subset \cup \overset{\_}{B}_{r_j}(x_j) \Big\}.}
	It is easy to see that $S$ is compact.
	Thus, there exists a point in $S$ that minimizes the continuous function
	\eq{L(x_1,x_2,\dots,x_M,r_1,\dots,r_M)=\sum_{j=1}^M r_j^{\alpha}.}
	This point corresponds to the desired optimal family $\mathfrak{B}$.
	
\end{proof}

\begin{proof}[\bf{Proof of Lemma~\ref{super1}.}]
	
	\
	
Let $\mathfrak{B}_0$ be a family of closed balls constructed in Lemma~\ref{opti}. The family $\mathfrak{B}$ will be some transformation of the family $\mathfrak{B}_0$. First, we split $\mathfrak{B}_0$ into families $\mathfrak{B}_1$ and $\mathfrak{B}_2$. The family $\mathfrak{B}_1$ consists of balls with radii not less than $\frac{\eps}{9}$, and the family $\mathfrak{B}_2$ consists of balls with radii less than $\frac{\eps}{9}$. We transform the families $\mathfrak{B}_1$ and $\mathfrak{B}_2$ into the families $\mathfrak{C}_1$ and $\mathfrak{C}_2$ and set $\mathfrak{B}=\mathfrak{C}_1\cup\mathfrak{C}_2$. Let $K_0$ be defined by the following formula:
\eq{K_0=K\cap\underset{\overset{\_}{B}_{r_i}(x_i)\in\mathfrak{B}_1}{\bigcup}\overset{\_}{B}_{r_i}(x_i).}
Let $E$ be a maximal $\frac{\eps}{3}$-separated subset of $K_0$. Let $\mathfrak{C}_1$ be defined by the formula
\eq{\mathfrak{C}_1=\{B_{\eps}(x)| \ x\in E\}.}
The family $\mathfrak{C}_1$ is a supercovering of the set $K_0$. The family $\mathfrak{C}_2$ is defined by the formula
\eq{\mathfrak{C}_2=\{B_{3r}(x)| \ \overset{\_}{B}_r(x)\in \mathfrak{B}_2\}.}

The family $\mathfrak{C}_2$ is a supercovering of the set $K\setminus K_0$. We will split $\mathfrak{C}_1$ and $\mathfrak{C}_2$ into subfamilies $\mathfrak{B}^j$ and this will complete the proof. It is clear that $\mathfrak{C}_1$ can be split into disjoint subfamilies $\mathfrak{B}^1,...,\mathfrak{B}^{100^d}$ such that this complete
\eq{\sum\limits_{B_{r_i}(x_i)\in\mathfrak{B}^j}r_i^{\alpha}\leqslant\sum\limits_{B_{r_i}(x_i)\in\mathfrak{C}_1}r_i^{\alpha}=|\mathfrak{C}_1|\eps^{\alpha}=|E|\eps^{\alpha}.}
The set $E$ is an $\frac{\eps}{3}$-separated set, so $|E\cap \overset{\_}{B}_{\eps}(x)|\leqslant 100^d$ for any $x\in\mathbb{R}^d$. Since $E$ is a subset of $\underset{\overset{\_}{B}_r(x)\in\mathfrak{B}_1}{\bigcup}\overset{\_}{B}_r(x)$, we also have $|E|\leqslant 100^d|\mathfrak{B}_1|$. With these inequalities, we may continue the estimate: 
\eq{\sum\limits_{B_{r_i}(x_i)\in\mathfrak{B}^j}r_i^{\alpha}\leqslant|E|\eps^{\alpha}\leqslant100^d|\mathfrak{B}_1|\eps^{\alpha}\leqslant 100^d9^d\sum\limits_{\overset{\_}{B}_{r_i}(x_i)\in\mathfrak{B}_1}r_i^{\alpha}\leqslant 100^d9^da.}

The family $\mathfrak{C}_2$ will be split into $\mathfrak{B}^j$ for $j>100^d$. We will also define the families $\mathfrak{B}^{>j}$ inductively. Let $\mathfrak{B}^{>100^d}=\mathfrak{C}_2$. Assume we have defined family $\mathfrak{B}^{>j}$. Then by Vitali's~Lemma we can find its disjoint subfamily $\mathfrak{B}^{j+1}$ (if we multiply all radii in family $\mathfrak{B}^{j+1}$ by $3$, it will cover family $\mathfrak{B}^{>j}$) and define the family $\mathfrak{B}^{>j+1}=\mathfrak{B}^{>j}\setminus\mathfrak{B}^{j+1}$. Let $q=1-9^{-\alpha}$. We will prove the inequality 
\eq{\sum\limits_{B_{r_i}(x_i)\in\mathfrak{B}^{>j+1}}r_i^{\alpha}\leqslant q\sum\limits_{B_{r_i}(x_i)\in\mathfrak{B}^{>j}}r_i^{\alpha}.}

If we enlarge all the radii of the balls in $\mathfrak{B}^{j+1}$ 3 times, the obtained family will cover all the balls in $\mathfrak{B}^{>j+1}$. The ball $B_r(x)\in\mathfrak{B}^{>j}$  
has its analogue $\overset{\_}{B}_{\frac{r}{3}}(x)$ in $\mathfrak{B}_0$. Let the family $\mathfrak{B}'$ be defined by the formula
\eq{\mathfrak{B}'=\{\overset{\_}{B}_r(x)| \ \overset{\_}{B}_r(x)\in\mathfrak{B}_0 \ \text{and} \ B_{3r}(x)\notin\mathfrak{B}^{>j}\}\cup\{\overset{\_}{B}_{3r}(x)|B_{r}(x)\in\mathfrak{B}^j\}.} 

The family $\mathfrak{B}'$  satisfies conditions 1), 2), 3)  of Lemma~\ref{opti}, so we have

\eq{\sum_{\overset{\_}{B}_{r_i}(x_i)\in\mathfrak{B}_0}r_i^{\alpha}\leqslant\sum_{\overset{\_}{B}_{r_i}(x_i)\in\mathfrak{B}'}r_i^{\alpha}.}
  
Consequently,
\eq{\sum_{B_{r_i}(x_i)\in\mathfrak{B}^{>j}}\left(\frac{r_i}{3}\right)^{\alpha}\leqslant\sum_{B_{r_i}(x_i)\in\mathfrak{B}^j}(3r_i)^{\alpha},}
\eq{\sum_{B_{r_i}(x_i)\in\mathfrak{B}^{>j+1}}r_i^{\alpha}=\sum_{B_{r_i}(x_i)\in\mathfrak{B}^{>j}}r_i^{\alpha}-\sum_{B_{r_i}(x_i)\in\mathfrak{B}^j}r_i^{\alpha}\leqslant q\sum_{B_{r_i}(x_i)\in\mathfrak{B}^{>j}}r_i^{\alpha},}

If we iterate the last inequality, we will have

\eq{\sum_{B_{r_i}(x_i)\in\mathfrak{B}^{>j}}r_i^{\alpha}\leqslant q^{j-100^d}\sum_{B_{r_i}(x_i)\in\mathfrak{B}^{>100^d}}r_i^{\alpha}.}

So, for $j>100^d$,
\begin{multline}
\sum_{B_{r_i}(x_i)\in\mathfrak{B}^j}r_i^{\alpha}\leqslant\sum_{B_{r_i}(x_i)\in\mathfrak{B}^{>j-1}}r_i^{\alpha}\leqslant q^{j-100^d-1}\sum_{B_{r_i}(x_i)\in\mathfrak{B}^{>100^d}}r_i^{\alpha}\leqslant\\ q^{j-100^d-1}3^{\alpha}\sum_{B_{r_i}(x_i)\in\mathfrak{B}_0}r_i^{\alpha}\leqslant q^{j-100^d-1}3^{\alpha}a.
\end{multline}
Let $C=q^{-100^d-1}3^{\alpha}100^d9^d$. Then

\eq{\sum_{B_{r_i}(x_i)\in\mathfrak{B}^j}r_i^{\alpha}\leqslant Cq^{j}a.}
  
  \end{proof}
\begin{Def}
	Set $B_R(x)\setminus B_{R-r}(x)$ will be called a ring with center~$x$, radius~$r$, and size~$R$. 
\end{Def}
Let $\omega_{d-1}$ be the area of the unit sphere in $\mathbb{R}^d$ and $\pi_d$ be the volume of the unit ball in $\mathbb{R}^d$.  
\begin{Le}\label{ring}
	Let $F$ be a ring with size $R$ and radius $r$. Then the volume of $F$ does not exceed $\omega_{d-1}R^{d-1}r$.
\end{Le}

\begin{proof}
We will integrate 1 over $F$ and make the spherical change of coordinates,
\eq{\int_{F}1dx=\int_{R-r}^{R}\omega_{d-1}t^{d-1}dt\leqslant\int_{R-r}^{R}\omega_{d-1}R^{d-1}dt=\omega_{d-1}R^{d-1}r.}	
\end{proof}
\begin{proof}[\bf{Proof of Lemma~\ref{super2}.}]
	
	\
	
Let $\mathfrak{B}^j_0$ be the families provided by Lemma~\ref{Morse1}. Every ball in these families will be transformed into not more than countable number of balls by the following algorithm. Pick number $q\in(0,1)$. Assume we have a ball $B_r(x)$. We split the ring $B_r(x)\setminus B_{\frac{r}{3}}(x)$ into countable number of rings $F^j(B_r(x))$ whose radii are decreasing like geometric progression. Let $r_j=\frac{r}{3}+\frac{2r(1-q)}{3}\sum_{k=0}^jq^k$, then
\eq{F^j(B_r(x))= B_{r_j}(x)\setminus B_{r_{j-1}}(x),}
\eq{B_r(x)\setminus B_{\frac{r}{3}}(x)=\underset{j=0}{\overset{\infty}{\bigcup}}F^j(B_r(x)).}
If $l=\frac{2(1-q)}{3}$, then the radius of $F^j(B_r(x))$ is $rlq^j$. Let $E^j(B_r(x))$ be the maximal $\frac{rlq^{j+1}}{3}$-separated subset of $F^j(B_r(x))\cap A$. If we put in every point of $E^j(B_r(x))$ a ball with center at this point and radius $\frac{rlq^{j+1}}{6}$ we will have disjoint family of balls. All balls in this family are contained in $\frac{rlq^{j+1}}{6}$-neighborhood of the ring $F^j(B_r(x))$. This neighborhood is inside the ring with center $x$ radius $rlq^j+2\frac{rlq^{j+1}}{6}<2rlq^j$, and size $\frac{rlq^{j+1}}{6}+r_j<r$. The sum of the volumes of balls does not exceed the volume of the ring, so
\eq{|E^j(B_r(x))|\pi_d \left(\frac{rlq^{j+1}}{6}\right)^d\leqslant 2\omega_{d-1}lq^jr^d,}
\begin{equation}\label{kolsa3}
|E^j(B_r(x))|\leqslant C_0q^{(1-d)j},
\end{equation}
where $C_0=\frac{2\omega_{d-1}l}{\pi_d(\frac{lq}{6})^d}$.
\begin{figure}[h]
	\center{\includegraphics[scale=0.5]{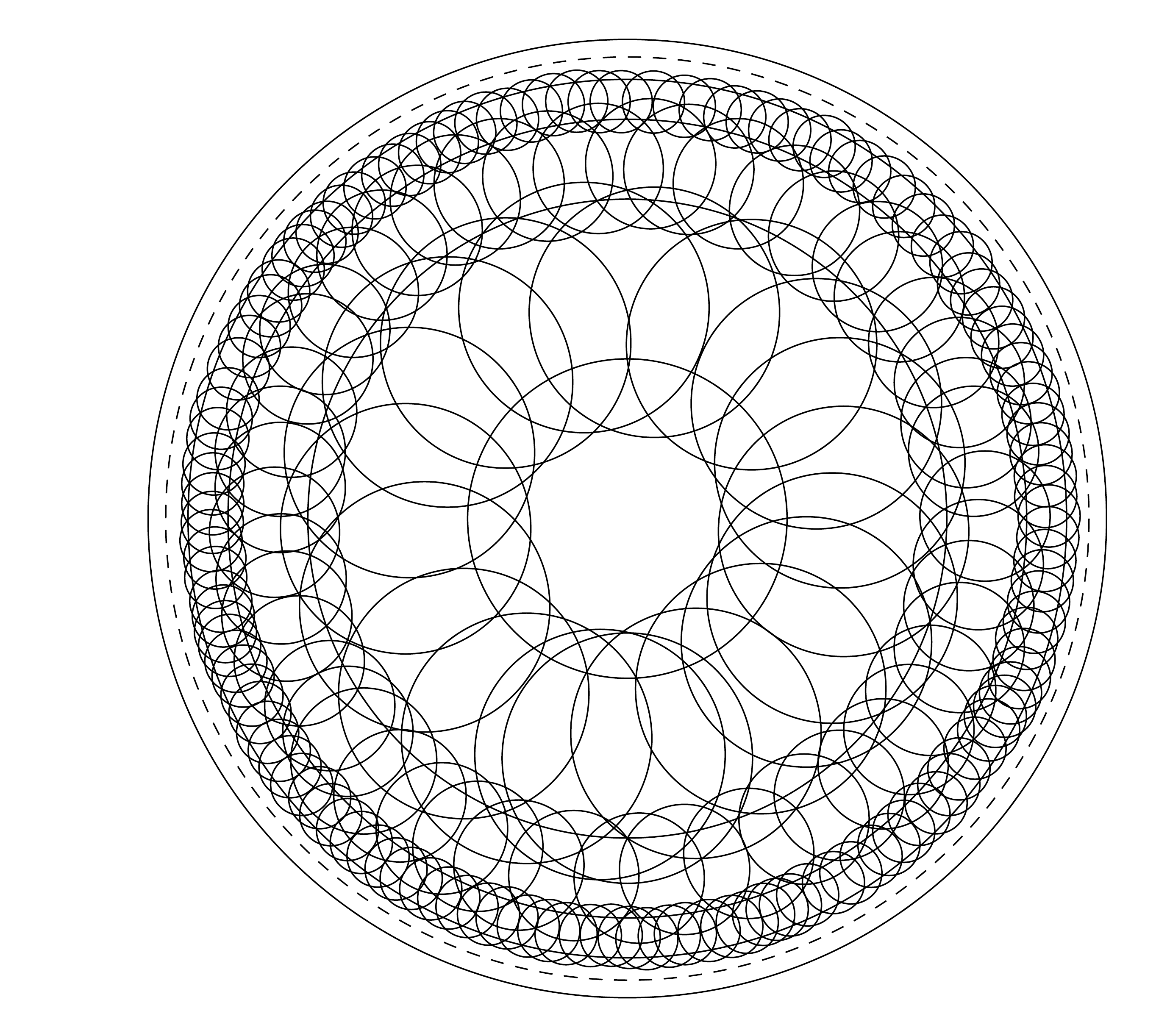}}
	\caption{Families $\mathfrak{C}^j(B_r(x))$.}\label{pic}
\end{figure}
The family $\mathfrak{C}^j(B_r(x))$ contain the balls with radii $rlq^{j+1}$ and centers in the set $E^j(B_r(x))$ (see figure~\ref{pic}). It is clear that the family $\mathfrak{C}^j(B_r(x))$ can be split into disjoint subfamilies $\mathfrak{C}^j_1(B_r(x))$,...,$\mathfrak{C}^j_{100^d}(B_r(x))$. The balls in family $\mathfrak{C}^j(B_r(x))$ does not intersect the rings $F^i(B_r(x))$ if $|i-j|\geqslant 2$, so the balls in family $\mathfrak{C}^j(B_r(x))$ do not intersect the balls in family $\mathfrak{C}^i(B_r(x))$ if $|i-j|\geqslant 3$. Let the family $\mathfrak{A}_{i,j}(B_r(x))$ for $0\leqslant i \leqslant 2$ and $1\leqslant j\leqslant 100^d$ be defined by the formula
\eq{\mathfrak{A}_{i,j}(B_r(x))=\underset{k=0}{\overset{\infty}{\bigcup}}\mathfrak{C}^{3k+i}_j(B_r(x)).}
We can write the inequality
\begin{equation}\label{ooo}
\sum_{B_{\tau_l}(y_l)\in\mathfrak{A}_{i,j}(B_r(x))}{\tau_l}^{\alpha}\overset{\eqref{kolsa3}}{\leqslant}\underset{k=0}{\overset{\infty}{\sum}}C_0q^{-k(d-1)}(lrq^{k+1})^{\alpha}=Cr^{\alpha},
\end{equation}
where $C=C_0l^{\alpha}q^{\alpha}\underset{k=0}{\overset{\infty}{\sum}}q^{k(\alpha-d+1)}$. Here we use that $\alpha\in (d-1,d]$.
Now we are ready to define a supercovering $\mathfrak{B}$ of the set $A$:

\eq{\mathfrak{B}=\left(\underset{m}{\bigcup}\underset{B_{r_k}(x_k)\in \mathfrak{B}^{m}_0}{\bigcup}\underset{i,j}{\bigcup}\mathfrak{A}_{i,j}(B_{r_k}(x_k))\right)\bigcup\left(\underset{j}{\bigcup}\mathfrak{B}^j_0 \right).}
We split $\mathfrak{B}$ into subfamilies $\mathfrak{B}^{m,i,j}$ and $\mathfrak{B}^j_0$, where 
\eq{\mathfrak{B}^{m,i,j}=\underset{B_{r_k}(x_k)\in \mathfrak{B}^{m}_0}{\bigcup}\mathfrak{A}_{i,j}(B_{r_k}(x_k)),}

and finish the proof with the estimate

\eq{\sum_{B_{r_k}(x_k)\in\mathfrak{B}^{m,i,j}}r_k^{\alpha}=
\sum_{B_{r_k}(x_k)\in \mathfrak{B}^{m}_0} \ \sum_{B_{\tau_l}(y_l)\in\mathfrak{A}_{i,j}(B_r(x))}{\tau_l}^{\alpha}\overset{\eqref{ooo}}{\leqslant}
 C\sum_{B_{r_k}(x_k)\in \mathfrak{B}^{m}_0}r_k^{\alpha}<C_2a.}

\end{proof}
\section{Proofs of lemmas and Theorem}\label{S3}
\begin{Le}\label{PPP}
	Let $\mu$ be a signed measure, let $A_{+}$ and $A_{-}$ be the sets of its Hahn decomposition, let $\mu_{+}$ and $\mu_{-}$ be its positive and negative parts. Consider the set 
	\eq{P_{+,\varepsilon}=\Big\{x\in A_{+}\mid\exists \delta(x) \ such \ that \  \forall r<\delta(x) \ \  \mu_{-}(B_r(x))\leqslant\varepsilon\mu_{+}(B_r(x))\Big\}.}
	Then $\mu_{+}(A)=\mu_{+}(P_{+,\varepsilon})$.
\end{Le}

Consider the set $P_{+,\varepsilon}^{(N)}$ given by formula
\eq{P_{+,\varepsilon}^{(N)}=\Big\{x\in A_{+}\mid\  \forall r<\frac{1}{N} \ \  \mu_{-}(B_r(x))\leqslant\varepsilon\mu_{+}(B_r(x))\Big\}.}
\begin{Le}
	 Let $x\in P_{+,\varepsilon}^{(N)}$ and let $r<\frac{1}{N}$. Then
	 \eq{\int \varphi\left(\frac{y-x}{r} \right)d\mu_{-}(y)\leqslant \varepsilon\int \varphi\left(\frac{y-x}{r} \right)d\mu_{+}(y)} 
	 for any radial non-negative test-function $\varphi$ supported in $B_1(0)$ that decreases as the radius grows.
\end{Le}
\begin{Le}\label{noc}
	Let $x\in P_{+,\varepsilon}^{(N)}$ and let $r<\frac{1}{N}$. Suppose that $\varphi$ is a radial non-negative function supported in a unit ball that
	decreases as the radius grows and let $\psi$ be a function such that $\varphi\leqslant \psi\leqslant \frac{1}{2\varepsilon}\varphi$. Then
	\eq{\int \psi\left(\frac{y-x}{r} \right)d\mu_{+}(y)\leqslant 2\int \psi\left(\frac{y-x}{r} \right)d\mu(y).}
\end{Le}
\begin{proof}
	We will write similar estimates:
	\eq{\int \psi\left(\frac{y-x}{r} \right)d\mu_{-}(y)\leqslant\frac{1}{2\varepsilon}\int \varphi\left(\frac{y-x}{r} \right)d\mu_{-}(y)\leqslant \frac{1}{2}\int \psi\left(\frac{y-x}{r} \right)d\mu_{+}(y).}
	So we can write
	\eq{\underset{\mathbb{R}^d}{\int}\psi\left(\frac{y-x}{r}\right)d\mu(y)=\underset{\mathbb{R}^d}{\int}\psi\left(\frac{y-x}{r}\right)d\mu_+(y)-\underset{\mathbb{R}^d}{\int}\psi\left(\frac{y-x}{r}\right)d\mu_-(y)\geqslant\frac{1}{2}\underset{\mathbb{R}^d}{\int}\psi\left(\frac{y-x}{r}\right)d\mu_+(y).}
\end{proof}
\paragraph{Proof of Lemma~\ref{teor1}.}
Without loss of generality we can say that $\varphi\leqslant \psi\leqslant C\varphi$.
Suppose that $A$ is a set such that $\mathcal{H}^{\alpha}(A)<a$, we will prove that $|\mu|(A)\lesssim h(a)$. Due to Lemma~\ref{PPP} it is enough to prove that $\mu_+(P_{+,\frac{1}{2C}})\lesssim h(a)$. Note that $P_{+,\frac{1}{2C}}=\cup P_{+,\frac{1}{2C}}^{(N)}$, so $\mu_{+}(P_{+,\frac{1}{2C}})\leqslant 2\mu_{+}(P_{+,\frac{1}{2C}}^{(N)})$ for $N$ large enough. Let $K$ be a compact subset of $P_{+,\varepsilon}^{(N)}$ such that $\mu_{+}(P_{+,\varepsilon}^{(N)})\leqslant2\mu_{+}(K)$. We will prove that $\mu_{+}(K)\lesssim h(a)$. Let $\mathfrak{B}$ be a supercovering of set $K$ provided by Lemma~\ref{super1}, and let $\mathfrak{B}^j$ be the corresponding subfamilies of $\mathfrak{B}$. We can write
\begin{multline}
	\mu_{+}(K)\lesssim
	\sum_{B_{r_i}(x_i)\in\mathfrak{B}}\int \psi\left(\frac{y-x_i}{r_i} \right)d\mu_{+}(y)\overset{lem~\ref{noc}}{\lesssim} 
	\sum_{B_{r_i}(x_i)\in\mathfrak{B}}\int \psi\left(\frac{y-x_i}{r_i} \right)d\mu(y)=\\
	\sum_{j=1}^{M}\sum_{B_{r_i}(x_i)\in\mathfrak{B}^j}\int \psi\left(\frac{y-x_i}{r_i} \right)d\mu(y)\lesssim\sum_{j=1}^{M}g\left(\sum_{B_{r_i}(x_i)\in\mathfrak{B}^j}r_i^{\alpha}\right)\leqslant\sum_{j=1}^{M}g(Cq^ja)\lesssim\sum_{j=1}^{M}g(q^ja)\lesssim\\ \int\limits_{0}^{a}\frac{g(t)}{t}dt=h(a).
\end{multline}
\paragraph{Proof of Lemma~\ref{teor2}.}
Let $A$ be the same stats in previous proof and let $A_{+}$ be the positive part of it. Consider the compact set $K$ such that $K\subset A_{+}$ and $\mu_{+}(A_{+})\leqslant 2\mu_{+}(K)$. Let $V$ be a open set such that $A_{+}\subset V$ and $\mu_{-}(V)\leqslant \varepsilon$. We will prove that $\mu_{+}(K)\lesssim g(a)$. Let $\mathfrak{B}$ be a supercovering of $K$ provided by Lemma~\ref{super2}. We may assume that any ball from $\mathfrak{B}$ lies in $V$. We can write
\begin{multline}
	\mu_+(K)\lesssim
	\sum_{B_{r_i}(x_i)\in\mathfrak{B}}\int\varphi\left(\frac{y-x_i}{r_i} \right)d\mu_{+}(y)=\\
	\sum_{B_{r_i}(x_i)\in \mathfrak{B}}\int\varphi\left(\frac{y-x_i}{r_i} \right)d\mu(y)+\sum_{B_{r_i}(x_i)\in \mathfrak{B}}\int\varphi\left(\frac{y-x_i}{r_i} \right)d\mu_{-}(y) \lesssim\\
	\sum_{j=1}^{C_1}\sum_{B_{r_i}(x_i)\in\mathfrak{B}^{j}}\int\varphi\left(\frac{y-x_i}{r_i} \right)d\mu(y)+C_1\mu_{-}(V)\lesssim \sum_{j=1}^{C_1}g\left(\sum_{B_{r_i}(x_i)\in\mathfrak{B}^j}r_i^{\alpha}\right)+C_1\varepsilon\leqslant\\
	 C_1g(C_2a)+C_1\varepsilon\lesssim g(a)+\varepsilon\underset{\varepsilon\rightarrow0}{\rightarrow}g(a).
\end{multline}

\paragraph{Proof of Lemma~\ref{teor3}.}
Consider the case $\varphi(x)=o(|x|^{-\alpha})$ first.

Let $\vartheta$ be a function such that $\varphi(x)=\vartheta(|x|)$. Let $A$ be a set such that $\mathcal{H}(A)<a$ and let $K$ be a compact set like in the proof of Lemma~\ref{teor1}. Let $\mathfrak{B}$ be a covering of $K$ provided by Lemma~\ref{super} such that the radius of any ball does not exceed $\frac{\varepsilon}{N}$. Note that for $r<\frac{\varepsilon}{N}$ we have
\eq{\int\limits_{\mathbb{R}^d\setminus B_{\frac{1}{N}}(x)}\psi\left(\frac{y-x}{r}\right)d|\mu|(y)\lesssim |\mu|(\mathbb{R}^d)r^{\alpha}\frac{\vartheta(\frac{1}{Nr})}{r^{\alpha}}\leqslant r^{\alpha}\delta(\varepsilon).}
Where $\delta(\varepsilon)=|\mu|(\mathbb{R}^d)\sup\limits_{r\leqslant\frac{\varepsilon}{N}}\frac{\vartheta(\frac{1}{Nr})}{r^{\alpha}}$. Note that $\delta(\varepsilon)\underset{\varepsilon\rightarrow0}{\rightarrow}0$  because $\varphi(x)=o(|x|^{-\alpha})$. We can write
\begin{multline}
	\mu_{+}(K)\lesssim
	\sum_{B_{r_i}(x_i)\in\mathfrak{B}}\int\limits_{B_{r_i}(x_i)}\psi\left(\frac{y-x_i}{r_i}\right)d\mu_{+}(y)\leqslant
	\sum_{B_{r_i}(x_i)\in\mathfrak{B}}\int\limits_{B_{\frac{1}{N}}(x_i)}\psi\left(\frac{y-x_i}{r_i}\right)d\mu_{+}(y)\overset{lem~\ref{noc}}{\lesssim}\\
	\sum_{B_{r_i}(x_i)\in\mathfrak{B}}\int\limits_{B_{\frac{1}{N}}(x_i)}\psi\left(\frac{y-x_i}{r_i}\right)d\mu(y)\leqslant\\
	\sum_{B_{r_i}(x_i)\in\mathfrak{B}}\int\limits_{\mathbb{R}^d}\psi\left(\frac{y-x_i}{r_i}\right)d\mu(y)+
	\sum_{B_{r_i}(x_i)\in\mathfrak{B}}\left|\int\limits_{\mathbb{R}^d\setminus B_{\frac{1}{N}}(x_i)}\psi\left(\frac{y-x_i}{r_i}\right)d\mu(y)\right|\lesssim\\
	\sum_{j=1}^{M}\sum_{B_{r_i}(x_i)\in\mathfrak{B}^j}\int\limits_{\mathbb{R}^d}\psi\left(\frac{y-x_i}{r_i}\right)d\mu(y)+\sum_{B_{r_i}(x_i)\in\mathfrak{B}}r_i^{\alpha}\delta(\varepsilon)\lesssim 
	\sum_{j=1}^{M}g\left(\sum_{B_{r_i}(x_i)\in\mathfrak{B}^j}r_i^{\alpha}\right)+a\delta(\varepsilon)\lesssim\\
	g(a)+a\delta(\varepsilon)\underset{\varepsilon\rightarrow0}{\rightarrow}g(a).
\end{multline}
Second case, $\varphi(x)=O(|x|^{-\alpha})$.

Let $A$ be a set such that $\mathcal{H}^{\alpha}(A)=0$. We will prove that $\mu_{+}(A)=0$. Let $K$ be a compact set like in the previous proof. For any positive $a$ we can write $\mathcal{H}^{\alpha}(A)<a$. If we make the same estimates we will have the inequality $\mu_{+}(K)\lesssim g(a)+a\delta(\varepsilon)\underset{a\rightarrow0}{\rightarrow}0$.

\

Let $\varphi$ be a function in $C_0^{\infty}(\mathbb{R}^{d-1})$ supported in the unit ball.  For $x=(x_1,x_2,\ldots,x_d)\in\mathbb{R}^d$ we write $x_{[i]}$ for the $(d-1)$-dimensional vector that is obtained from $x$ by forgetting the $i$-th coordinate (for example, for $d = 3$, $x_{[2]} = (x_1, x_3)$).

We cite Lemma~2.3 in \cite{StolyarovWojciechowski2014}.
\begin{Le}\label{23}
	Let $\mathfrak{B}$ be a disjoint family of balls in $\mathbb{R}^{d-1}$, let $\psi\in C_0^{\infty}(\mathbb{R})$ be a test function. Suppose that $f$ is a compactly supported function. If $\mu=(\partial_1^{m_1} f, \partial_2^{m_2} f,\ldots, \partial_d^{m_d} f)$ is a measure, then, for all $i=1,2\ldots d$ any $\varphi\in C_0^{\infty}(\mathbb{R}^{d-1})$ supported in unit ball,
	\eq{\sum_{B_{r}(x)\in \mathfrak{B}}\left|\int\limits_{\mathbb{R}^d}\psi(x_i)\varphi\left(\frac{y-x}{r}\right)d\mu_i(y) \right|\lesssim\left(\sum_{B_r(x)\in\mathfrak{B}}r^{d-1}\right)^{\frac{1}{q_i'}}}
	for some fixed $q_i'$ (the constants may depend on $\varphi$ and $\psi$)
\end{Le}
\begin{Le}\label{25}
	Let $\mu$ be a Borel measure on $\mathbb{R}^{l+k}$. Suppose that $\mu(I\times A)=0$ for every parallelepiped $I\subset \mathbb{R}^k$ and every Borel $A\subset\mathbb{R}^{l}$ such that $\mathcal{H}^{\alpha}(A)=0$. Then $\mu$ is absolutely continuously with respect to $\mathcal{H}^{\alpha}$. 
\end{Le}
Proof of Lemma~\ref{25} is absolutely similar to the proof of Lemma 2.5 in~\cite{StolyarovWojciechowski2014}.
\paragraph{Proof of Theorem~\ref{Stol}}
Assume the contrary. Let $F$ be some Borel set such that $\mathcal{H}^{d-1}(F)=0$, but $\mu(F)\neq0$. We may assume that $\mu_1(F)\neq0$ (by symmetry) and $F$ is compact (due to the regularity of the measure). Multiplying $f$ by a test function that equals $1$ on $F$, we make $f$ compactly supported without loosing the condition that its higher order derivatives are signed measures. To get a contradiction, it suffices to prove that for every set $A\subset\mathbb{R}^{d-1}$ such that $\mathcal{H}^{d-1}(A)=0$ and every function $\psi\in C_0^{\infty}(\mathbb{R})$, we have:
\begin{equation}\label{fors}
	\int\limits_{A\times \mathbb{R}}\psi(x_1)d\mu_1(x)=0.
\end{equation}
Then, approximating the characteristic function of an interval $I$ by smooth functions, we get the hypothesis of Lemma~\ref{25} with $\alpha=d-1$, which, in its turn, asserts that $\mu_1(F)=0$.

Consider now a complex measure $\mu_{\psi}$ on $\mathbb{R}^{d-1}$ given by formula $\mu_{\psi}(B)=\int_{B\times\mathbb{R}}\psi(x_1)d\mu_1(x)$ and note that (\ref{fors}) holds for any $A$ such that $\mathcal{H}^{d-1}(A)=0$. By Lemma~\ref{23}, $\mu_{\psi}$ satisfies the hypothesis of Lemma~\ref{teor2} with $\alpha=d-1$. Therefore, $\mu_{\psi}$ is absolutely continuous with respect to $\mathcal{H}^{d-1}$.
\section{Generalizations and examples}

\begin{Le}\label{justification}
		Suppose that~$\varphi \in C^{\infty}_0(\mathbb{R}^d)$ is a radial non-negative function supported in a unit ball. Assume that~$\varphi$ decreases as the radius grows and~$\varphi(x) = 1$ when~$|x| \leq\frac34$.  
		Let $\mu$ be a positive Borel measure such that $\mu(B_r(x))\leqslant r^{\alpha}$ for any open ball $B_r(x)$. Let $\nu$ be a signed measure continuous with respect to $\mu$. Then for any disjoint family of balls $\mathfrak{B}$ the inequality
		\begin{equation}
			\sum_{B_{r_j}(x_j)\in\mathfrak{B}}\left|\int\varphi\left(\frac{y-x_j}{r_j}\right)d\nu(y)\right|\lesssim g\left(\sum_{B_{r_j}(x_j)\in\mathfrak{B}}r_j^{\alpha}\right)\text{ holds true},
		\end{equation}
		where
		\begin{equation}
		g(t)=\int\limits_{0}^{t}\left|\frac{d\nu}{d\mu}\right|^{*}(s)ds.
		\end{equation} 
	\end{Le}
	Here and in what follows the notation~$f^{*}$ means the monotonic rearrangement of function $f$ and $\frac{d\nu}{d\mu}$ is the density of $\nu$ with respect to $\mu$. The monotonic rearrangement of function $f$ can be defined by formula:
	\begin{equation}
	f^{*}(t)=\sup\Set{\alpha}{ \mu\Set{x}{f(x)>\alpha}\leqslant t}.
	\end{equation}
	
	\begin{proof}
		Let $\mathfrak{B}$ be a disjoint family of balls and let $\mathcal{B}=\underset{B_{r_j}(x_j)\in \mathfrak{B}}{\bigcup}B_{r_j}(x_j)$. Note that $\mu(\mathcal{B})\leqslant \sum_{B_{r_j}(x_j)\in\mathfrak{B}}r_j^{\alpha}$. We can write
		\begin{equation}
			\sum_{B_r(x)\in\mathfrak{B}}\left|\int\varphi\left(\frac{y-x}{r}\right)d\nu(y)\right|\lesssim |\nu|(\mathcal{B})\leqslant
			\int\limits_0^{\mu(\mathcal{B})}\left|\frac{d\nu}{d\mu}\right|^{*}(s)ds=g(\mu(\mathcal{B}))\leqslant g\left(\sum_{B_r(x)\in \mathfrak{B}}r^{\alpha}\right).
		\end{equation}
	\end{proof}

	 Our lemmas may be generalised for $f$ --- Hausdorf measures.
	 We remind the definition of this measures.
	 \begin{Def}
	 	Let $f$ be a regular function. The $f$ --- Hausdorff measure be defined by the formula  
	 	\begin{equation}
	 		\Lambda_f(A)=\lim\limits_{\delta\rightarrow 0} \underset{\diam(B_j)<\delta}{\underset{A\subset\cup B_j}{\inf}}\sum f(\diam(B_j)).
	 	\end{equation}
	\end{Def}
	Note that if $f(t)=t^{\alpha}$ then $\Lambda_f=\mathcal{H}^{\alpha}$. We will formulate more general versions of our lemmas.
	
	\begin{Le}\label{teor4}
		Let~$\varphi$ be a radially symmetric, radially non-increasing function supported in the unit ball. Assume also~$\varphi(x)=1$ when~$|x|\leqslant \frac{3}{4}$. Let~$\psi$ be equivalent to~$\varphi$. Let~$\mu$ be an~$\R$-valued signed measure of locally bounded variation, let~$g$ be a regular function that satisfies the Dini condition
		\eq{
			\int\limits_0^1\frac{g(t)}{t}dt<\infty.
		}
		Assume that  
		\eq{
			\sum\limits_{B_{r_j}(x_j)\in \mathfrak{B}}\left|\underset{\mathbb{R}^d}{\int}\psi\left(\frac{y-x_j}{r_j} \right) d\mu(y) \right|\lesssim g\Big(\sum\limits_{B_{r_j}(x_j)\in \mathfrak{B}}f(r_j)\Big)
		}
		for any disjoint family of balls~$\mathfrak{B}$. Then,
		\eq{
			|\mu|(A)\lesssim h(\Lambda_f(A)),
		}
		where
		\eq{
			h(x)=\int\limits_0^x\frac{g(t)}{t}dt
		}
		and~$A$ is an arbitrary Borel set~$A \subset \R^d$.
	\end{Le}
	
	\begin{Def}
		A regular function $f$ is $d$ --- falling if it satisfies the conditions
		\begin{enumerate}[1)]
			\item 
				\begin{equation}
					\int\limits_0^1\frac{f(t)}{t^d}dt<\infty,
				\end{equation}
			\item 
				\begin{equation}
					x^{d-1}\int\limits_0^x\frac{f(t)}{t^d}dt\asymp f(x).
				\end{equation}
							
		\end{enumerate}
	\end{Def}
	
	\begin{Le}\label{teor5}
		Let~$\varphi$ be a bounded function supported in the unit ball and such that~$\varphi(x)\geqslant1$ when~$|x|\leqslant \frac{3}{4}$. Let~$\mu$ be an~$\R$-valued signed measure of locally bounded variation, let~$g$ be a regular function. Assume also that~$f$ is $d$ --- falling. If
		\eq{
			\sum\limits_{B_{r_j}(x_j)\in \mathfrak{B}}\left|\underset{\mathbb{R}^d}{\int}\varphi\left(\frac{y-x_j}{r_j} \right) d\mu(y) \right|\lesssim g\Big(\sum\limits_{B_{r_j}(x_j)\in \mathfrak{B}}f(r_j)\Big)
		}
		for any disjoint family of balls~$\mathfrak{B}$, then
		\eq{
			|\mu|(A)\lesssim g(\Lambda_f(A))
		}
		for any Borel set~$A \subset \R^d$.
	\end{Le}
	
	\begin{Le}\label{teor6}
		Let~$\varphi$ be a bounded radially symmetric, radially non-increasing function such that~$B_3(0)\subset\supp \varphi $. Let~$\psi$  be equivalent to~$\varphi$. Let~$\mu$ be a signed measure of bounded variation, let~$g$ be a regular function. Assume
		\eq{
			\sum\limits_{B_{r_j}(x_j)\in \mathfrak{B}}\left|\underset{\mathbb{R}^d}{\int}\psi\left(\frac{y-x_j}{r_j} \right) d\mu(y) \right|\lesssim g\Big(\sum\limits_{B_{r_j}(x_j)\in \mathfrak{B}}f(r_j)\Big)
		}
		for any disjoint family of balls~$\mathfrak{B}$. If
		\eq{
			|\varphi(x)|=O(f(|x|^{-1})), \ \ x\rightarrow \infty,
		}
		then~$\mu$ is absolutely continuous with respect to~$\Lambda_f$. If
		\eq{
			|\varphi(x)|=o(f(|x|^{-1})),\quad x\to \infty,
		}
		then the inequality
		\eq{
			|\mu|(A)\lesssim g(\Lambda_f(A))
		}
		holds true for any Borel set~$A$.
	\end{Le}
	
	The proofs of lemmas \ref{teor4}, \ref{teor5}, \ref{teor6} are exactly the same as the proofs of lemmas  \ref{teor1}, \ref{teor2}, \ref{teor3}.
\

\

\

\

\

\

{\small
	
	St. Petersburg State University Department of Leonhard Euler International Mathematical Institute;\\
	e-mail:dobronravov1999@mail.ru\\
	\bigskip

}

\end{document}